\documentclass{amsart}
\pdfoutput=1
\usepackage{times, enumerate, array, longtable, %
amssymb, diagrams, mathrsfs, bbm, times, txfonts, verbatim}

\title{Prime Filters in MV-algebras}
\author{Colin G. Bailey}
\address{School of Mathematics,  Statistics \& Operations Research\\
Victoria University of Wellington\\
Wellington, New Zealand\\
}
\email{Colin.Bailey@vuw.ac.nz}

\date{\today}

\def\one{{\mathbf 1}}

\def\eqcl[#1]{\pmb{[}#1\pmb{]}}

\def\leftGen{{[\kern-1.1pt[}}
\def\rightGen{{]\kern-1.1pt]}}

\let\rsf\mathscr
\providecommand{\meet}{\mathbin{\wedge}}
\providecommand{\join}{\mathbin{\vee}}

\ifx\upharpoonright\undefined
     \def\restrict{\hbox{\rm\kern0.166em\accent"12\kern-0.536em$\vert$\kern0.3em}}%
  \else
     \def\restrict{\upharpoonright}%
  \fi
\makeatletter
\def\twoSet#1#2{\left\{%
\vphantom{#2}#1\thinspace\right|\nolinebreak[3]\left.%
  #2%
  \vphantom{#1}%
  \right\}%
}
\def\oneSet#1{\left\lbrace#1\right\rbrace}

\newif\if@nstr
\def\setstrfalse{\let\if@nstr=\iffalse}
\def\setstrtrue{\let\if@nstr=\iftrue}
\def\@nstr #1#2{
\def\@@nstr ##1#1##2##3\@@nstr{\ifx
\@nstr ##2\setstrfalse \else \setstrtrue \fi }
\@@nstr #2#1\@nstr \@@nstr}
\def\@separate#1|#2@{\setFront{#1}\setBack{#2}}
\def\lb#1\rb{\@nstr|{#1} \if@nstr \@separate#1 @ \twoSet{\@setFront}{\@setBack}%
\else \@separate |{#1 }@ \oneSet{\@setBack}\fi%
}
\def\setFront#1{\def\@setFront{#1}}
\def\setBack#1{\def\@setBack{#1}}
\def\Set#1{\lb{#1}\rb}

\def\oneBrk#1{\left\langle#1\right\rangle}
\def\twoBrk#1#2{\left\langle%
\vphantom{#2}#1\thinspace\right|\nolinebreak[3]\left.%
  #2%
  \vphantom{#1}%
  \right\rangle%
}
\def\brk<#1>{\@nstr|{#1} \if@nstr \@separate#1 @ \twoBrk{\@setFront}{\@setBack}%
\else \@separate |{#1 }@ \oneBrk{\@setBack}\fi%
}

\def\thmref#1{\normalfont{theorem}~\ref{#1}}

\def\propref#1{\normalfont{proposition}~\ref{#1}}
\theoremstyle{plain}
\newtheorem{thm}{Theorem}[section]
\newtheorem{lem}[thm]{Lemma}
\newtheorem{cor}[thm]{Corollary}
\newtheorem{prop}[thm]{Proposition}
\newtheorem{defn}[thm]{Definition}

\theoremstyle{remark}
{}
{}
{}
{}

\@ifpackageloaded{bbm}{
\newcommand{\Q}{{\mathbbm{Q}}}

}{

\newcommand{\Q}{{\mathbb{Q}}}

}
\makeatother

\begin{document}
\begin{abstract} 	
   In this document we consider the prime spectrum of 
   an MV-algebra with certain natural operations. These are used to 
   show connections between the classes of prime lattice filters and 
   prime implication filters. 
\end{abstract}
\maketitle

\section{Introduction}
In the representation theorems for MV-algebras of Martinez 
(\cite{Martinez:one, Martinez:two}) and Martinez \& Priestley 
(\cite{MartPreist:one}) much use was made of 
the function on filters: $a\mapsto \rsf F_{a}=\Set{z | z\to a\notin\rsf F}$. For 
a fixed lattice filter $\rsf F$ the set $\Set{\rsf F_{a} | 
a\in\mathcal L}$ was shown to be linearly ordered. We 
begin by analysing this result further using the kernel of a filter 
(defined below) and (in a later paper) extend our analysis to define 
natural operations on filters that result in MV-algebras. This 
process generalizes the notion of cuts in $[0, 1]_{\Q}$. The 
operations defined are closely related to the operations on filters defined in 
\cite{Martinez:one, Martinez:two}.

As is usual when studying filters in MV-algebras we need to be aware 
of the type of filter -- is it merely an order filter,  or a lattice 
filter or an implication filter? Most of what follows explores an 
interaction between 
order filters and implication filters. 

\begin{defn}
    Let $\rsf F$ be any order filter.
    The set 
    $$
    \rsf F_{a}=\Set{z | z\to a\notin\rsf F}
    $$
    is called the \emph{subordinate of $\rsf F$ at $a$}.
\end{defn}

\begin{prop}
    Let $\rsf F$ be any order filter and $a\notin\rsf F$.
Then $\rsf F_{a}$ is 
    also an order filter that does not contain $a$.
    
    $\rsf F_{a}$ is prime if $\rsf F$ is meet-closed,  and 
    $\rsf F_{a}$ is meet-closed if $\rsf F$ is prime. 
\end{prop}
\begin{proof}
    $1\to a=a\notin\rsf F$ so that $1\in\rsf F_{a}$.
    
    If $z\geq w\in\rsf F_{a}$ then $z\to a\le w\to a\notin\rsf F$ and 
    so we must have $z\to a\notin\rsf F$.
    
    As $a\to a=1\in\rsf F$ we see that $a\notin\rsf F_{a}$.
    
    Suppose that $\rsf F$ is meet-closed. 
    If $x\join y\in\rsf F_{a}$ then $(x\join y)\to a= (x\to 
    a)\meet(y\to a)\notin\rsf F$. Hence at least one of $x\to a$ and 
    $y\to a$ cannot be in $\rsf F$ -- as it is meet-closed. Hence at 
    least one of $x$ or $y$ is in $\rsf F_{a}$.

    Suppose that $\rsf F$ is prime.
    If $x,y\in\rsf F_{a}$ then $(x\meet y)\to a= (x\to a)\join(y\to 
    a)$ is in $\rsf F$ iff $x\to a\in\rsf F$ or $y\to a\in\rsf F$ -- 
    as $\rsf F$ is prime.
    As neither of these is true, $x\meet y\in\rsf F_{a}$.
\end{proof}

We note that if $a\in\rsf F$ then $\rsf F_{a}=\emptyset$ as we have 
$z\to a\geq a\in\rsf F$ for all $z\in\mathcal L$.

\begin{prop}
    Let $a\le b$. Then $\rsf F_{b}\subseteq\rsf F_{a}$. 
\end{prop}
\begin{proof}
    As $z\to a\le z\to b$ for any $z$, if $z\to b\notin\rsf F$ then 
    neither is $z\to a$.
\end{proof}

This tells us that  the set $\Set{\rsf F_{a}| a\notin\rsf F}$ is 
directed up and down as if $a,b\notin\rsf F$ then 
$\rsf F_{a\join b}\subseteq \rsf F_{a}\cap\rsf F_{b}$ and 
$\rsf F_{a}\cup\rsf F_{b}\subseteq\rsf F_{a\meet b}$.

The proposition also tells us that there is a largest filter in the
family $\Set{\rsf 
F_{a} | a\notin\rsf F}$, namely $\rsf F_{0}$ which we henceforth 
denote by $\rsf F^{+}$. 

\begin{prop}\label{prop:subOrder}
    Let $\rsf F\subseteq\rsf G$ be two order filters with 
    $a\notin\rsf G$. Then
    $$
    \rsf G_{a}\subseteq\rsf F_{a}.
    $$
\end{prop}
\begin{proof}
    As $z\to a\notin\rsf G$ implies $z\to a\notin\rsf F$.
\end{proof}

\begin{prop}
    Let $\rsf F$ be any order filter and $a\notin\rsf F$. Then
    $$
    (\rsf F_{a})_{a}=\rsf F.
    $$
\end{prop}
\begin{proof}
    \begin{align*}
        z\in(\rsf F_{a})_{a} & \iff z\to a\notin\rsf F_{a}  \\
         & \iff (z\to a)\to a= z\join a\in\rsf F  \\
         & \iff z\in\rsf F
    \end{align*}
    as $a\notin\rsf F$ and $\rsf F$ is prime.
\end{proof}

As a special case we have the following facts about $\rsf 
F\mapsto\rsf F^{+}$.

\begin{cor}
    Let $\rsf F\subseteq\rsf G$ be two order filters. Then
    \begin{enumerate}
        \item $\rsf F^{+}$ is a prime lattice filter.
        \item $\rsf G^{+}\subseteq\rsf F^{+}$.
        \item $\rsf F^{++}=\rsf F$.
    \end{enumerate}
\end{cor}

It is also worthwhile noting that 
$\rsf F^{+}$ is equal to $(\mathcal L\setminus\rsf F)^{*}$ where 
$X^{*}=\Set{\lnot x | x\in X}$, and that $\rsf F$ a prime filter
implies $\mathcal L\setminus\rsf F$ is a prime ideal.

At the lower limit of all the subordinates of $\rsf F$ we have its 
kernel.
\begin{defn}
    The \emph{kernel} of an order filter $\rsf F$ is the set
    $$
    \mathcal{K}(\rsf F)=\Set{z | \forall a\notin\rsf F\ z\to 
    a\notin\rsf F}.
    $$
\end{defn}
Clearly $\mathcal K(\rsf F)=\bigcap_{a\notin\rsf F}\rsf F_{a}$. 

The kernel (in its dual form, for ideals) can be also seen in \cite{MainBk}. 

\section{Properties of $\mathcal K$}
\begin{thm}
    If $\rsf F$ is an order filter then $\mathcal K(\rsf F)$ 
    is an implication filter contained in $\rsf F$.
    
    If $\rsf F$ is a prime filter then so is $\mathcal K(\rsf F)$.
\end{thm}
\begin{proof}
    If $x\in\mathcal K(\rsf F)$ and $x\le y$ then for any 
    $a\notin\rsf F$ we have $y\to a\le x\to a\notin\rsf F$ and so 
    $y\to a\notin\rsf F$. Hence $y\in\mathcal K(\rsf F)$.
    
    If $x, y\in\mathcal K(\rsf F)$ and $a\notin \rsf F$ then $y\to 
    a\notin\rsf F$ and so $(x\otimes y)\to a= x\to(y\to a)\notin\rsf 
    F$. Hence $\mathcal K(\rsf F)$ is $\otimes$-closed. Therefore it 
    is also closed under $\meet$. 
    
    If $a\notin \rsf F$ then $a\notin\rsf F_{a}$ and so 
    $a\notin\mathcal K(\rsf F)$. Hence $\mathcal K(\rsf 
    F)\subseteq\rsf F$.
    
    If $x,y\notin\mathcal K(\rsf F)$ then we have $a_{x}$ and $a_{y}$ 
    not in $\rsf F$ with $x\to a_{x}\in\rsf F$ and $y\to 
    a_{y}\in\rsf F$. Then, taking $a=a_{x}\join a_{y}$ we see that 
    $a\notin\rsf F$ (as $\rsf F$ is prime) and $x\to a\in\rsf F$ and 
    $y\to a\in\rsf F$. Therefore $(x\join y)\to a= (x\to a)\meet 
    (y\to a)\in\rsf F$ and so $x\join y\notin\mathcal K(\rsf F)$.
\end{proof}

\begin{prop}
    If $\rsf F$ is an implication filter then $\rsf F=\mathcal 
    K(\rsf F)$.
\end{prop}
\begin{proof}
    Let $f\in\rsf F$. If $f\notin\mathcal K(\rsf F)$. Then there is
some $a\notin\rsf 
    F$ with $f\to a\in\rsf F$.  As $f\to a\in\rsf F$ and $f\in\rsf F$ 
    we have $a\in\rsf F$ -- contradiction. Hence $f\in\mathcal K(\rsf 
    F)$. 
\end{proof}

\begin{thm}\label{thm:OTClosed}
    Let $\rsf F$ be any order filter. Then
    $$
    \mathcal K(\rsf F)=\Set{ z | \forall f\in\rsf F\ f\otimes 
    z\in\rsf F}.
    $$
\end{thm}
\begin{proof}
    \begin{align*}
	z\in\mathcal K(\rsf F) &\iff \forall a\notin\rsf F\ z\to 
	a\notin\rsf F\\
	&\iff \forall a\notin\rsf F\forall f\in\rsf F \lnot(f\le z\to 
	a)\\
	&\iff \forall a\notin\rsf F\forall f\in\rsf F\lnot(f\otimes 
	z\le a)\\
	&\iff\forall f\in\rsf F\lnot (f\otimes z\in\mathcal 
	L\setminus \rsf F)\\
	&\iff \forall f\in\rsf F f\otimes z\in\rsf F.
    \end{align*}
\end{proof}

It is immediate from this theorem that
\begin{cor}
    $$
	\mathcal K([p, 1])=\Set{ q | q\otimes p=p}.
    $$
\end{cor}
We need to later use this for linearly ordered MV-algebras in which 
case further simplification occurs.

\begin{prop}
    If $\mathcal L$ is linearly ordered and $p>0$ then 
    $$
    \mathcal K([p, 1])=\Set1.
    $$
\end{prop}
\begin{proof}
    Let $\lnot p<q<1$. We will show that $q\otimes p<p$. Obviously 
    there is no need to consider $q\le\lnot p$.
    
    Consider the set $p\to[0, p]= \Set{p\to r | r\le p}= [\lnot p, 1]$.
    This contains $q$ and so there is an $r\le p$ with $p\to r=q$. As 
    $q<1$ we know $r<p$. But now
    $r= p\otimes(p\to r)= p\otimes q< p$. 
\end{proof}

$\mathcal K(\rsf F)$ has a very special property with respect to quotients.
For any implication filter $Q$ we let $\eta_{Q}\colon\mathcal 
L\to\mathcal L/Q$ denote the canonical epimorphism, and
$X/Q$ the image of any subset of $\mathcal L$.
\begin{thm}\label{thm:inclK}
    Let $\rsf F$ be any order filter, and $\rsf P$ any 
    implication filter. Then
    $$
    \eta_{P}^{-1}[\rsf F/P]=\rsf F \text{ iff }P\subseteq\mathcal 
    K(\rsf F).
    $$
\end{thm}
\begin{proof}
    Suppose that $\eta_{P}^{-1}[\rsf F/P]=\rsf F$ and $f\in\rsf F$, 
    $p\in P$. Then
    $\eta_{P}(p\otimes x)=\eta_{P}(x)\otimes \eta_{P}(f)= 
    \eta_{P}(f)\in\rsf F/P$. Hence $p\otimes x\in\eta_{P}^{-1}[\rsf 
    F/P]=\rsf F$. By the \thmref{thm:OTClosed} we have $P\subseteq\mathcal 
    K(\rsf F)$. 
    
    Conversely, suppose $P\subseteq\mathcal K(\rsf F)$ is an 
    implication filter. Let $h\in\eta_{P}^{-1}[\rsf F/P]$. Then there 
    is some $f\in\rsf F$ with $f\to h\in P$ and $h\to f\in P$. Thus 
    $f\to h\in\mathcal K(\rsf F)$.  If $h\notin\rsf F$ then
    $f\to h\in\rsf F_{h}$ and so 
    $(f\to h)\to h= f\join h\notin\rsf F$ -- which is absurd as 
    $f\in\rsf F$. Thus $\eta_{P}^{-1}[\rsf F/P]\subseteq\rsf F$. The 
    other direction is immediate.
\end{proof}

\section{Interactions between $\mathcal K$ and subordination}
We begin by computing the kernel of $\rsf F^{+}$. In what follows we 
are primarily interested in prime filters,  although some results 
hold more generally.

\begin{prop}\label{prop:filPlus}
    Let $\rsf F$ be any order filter. Then
    $$
    \mathcal K(\rsf F)=\mathcal K(\rsf F^{+}).
    $$
\end{prop}
\begin{proof}
    Let $z\in\mathcal K(\rsf F)$. Now 
    \begin{align*}
	w\in\mathcal K(\rsf F^{+})\text{ iff }& a\notin\rsf F^{+}\text{
implies }w\to 
    a\notin\rsf F^{+}\\
    \text{ iff }&\lnot a\in\rsf F\text{ implies }w\otimes\lnot 
    a\in\rsf F\\
    \text{ iff }& w\otimes\lnot a\notin\rsf F\text{ implies }\lnot 
    a\notin\rsf F\\
    &\text{ iff }\lnot a\in\rsf F\text{ implies }w\otimes\lnot 
    a\in\rsf F\\
    &\text{ iff }b\in\rsf F\text{ implies }w\otimes b\in\rsf F\\
    &\text{ iff }b\in\mathcal K(\rsf F).
\end{align*}
%
\end{proof}

We want to extend this to all subordinates. Our approach is indirect, 
we work in the interval $[a,1]$ to study $\rsf F_{a}$ -- using the 
induced MV-structure on this interval. However we need some idea of 
the relationship between the $\mathcal K_{[a,1]}$ and $\mathcal 
K=\mathcal K_{\mathcal L}$. 

\begin{defn}
    Let $\rsf F$ be any order filter. Then $\rsf F^{\geq a}=\rsf 
    F\cap[a,1]$ is the \emph{localization of $\rsf F$ to $[a,1]$}.
\end{defn}

\begin{prop}
    Let $\rsf F$ be any order filter. Let $a\le b$ in $\mathcal L$. 
    Let $\rsf G=\left(\rsf F_{b}\right)^{\geq a}$ and 
    $\rsf H$ be $\left(\rsf F^{\geq a}\right)_{b}$ as computed in 
    $[a,1]$. 
    
    Then $\rsf G=\rsf H$.
\end{prop}
\begin{proof}
    If $y\in\rsf H$ then $y\to b\geq b\geq a$ and so 
    $y\to b\notin\rsf F$ iff $y\to b\notin \rsf F^{\geq a}$. Hence 
    $y\in\rsf G$.
    
    If $y\in\rsf G$ then $y\geq a$ and $y\to b\notin\rsf F$ and so 
    $y\in\rsf H$.
\end{proof}

\begin{prop}
    If $\rsf F$ is any prime filter and $a\notin\rsf F$ then
    $$
    \mathcal K_{[a,1]}(\rsf F^{\geq a})=\mathcal K(\rsf F)^{\geq a}.
    $$
\end{prop}
\begin{proof}
    From the last proposition we know that if $b\in[a,1]\setminus\rsf 
    F^{\geq a}$ then $b\notin\rsf F$ and $\left(\rsf F^{\geq 
    a}\right)_{b}=\rsf F_{b}\cap[a,1]$. Hence 
    \begin{align*}
	\mathcal K_{[a,1]}(\rsf F^{\geq 
	a})&=\bigcap_{b\in[a,1]\setminus\rsf F^{\geq 
	a}}\left(\rsf F^{\geq a}\right)_{b}\\
	&= \left(\smash[b]{\bigcap_{b\in[a,1]\setminus\rsf F}}\rsf 
	F_{b}\right)\cap[a,1]\\[10pt]
	&\supseteq \left(\smash[b]{\bigcap_{b\notin\rsf F}}\rsf 
	F\right)\cap[a,1]\\[10pt]
	&= \mathcal K(\rsf F)\cap[a,1].
    \end{align*}
    Conversely, if $z\geq a$ and $z\notin\mathcal K(\rsf F)$ then 
    there is some $w\notin\rsf F$ with $z\to w\in\rsf F$.  Hence 
    $z\to (w\join a)= (z\to w)\join(z\to a)\in\rsf F\cap[a,1]$, ie
    $z\notin \rsf F_{w\join a}\cap[a,1]$. As $\rsf F$ is prime, 
    $w\join a\notin\rsf F$ and so $z\notin \mathcal 
    K_{[a,1]}(\rsf F^{\geq a})$.
\end{proof}

Now we notice that from the point of view of $[a,1]$ we have $\left(\rsf F^{\geq 
a}\right)^{+}$ is exactly $\left(\rsf F^{\geq a}\right)_{a} = \rsf 
F_{a}^{\geq a}$. Hence we have
\begin{lem}
    $$
    \mathcal K_{[a,1]}(\rsf F^{\geq a})=\mathcal K_{[a,1]}(\rsf 
    F_{a}^{\geq a}).
    $$
\end{lem}
\begin{proof}
    Apply \propref{prop:filPlus}.
\end{proof}

And so we have 
$$
\mathcal K(\rsf F)^{\geq a}=\mathcal K_{[a,1]}(\rsf F^{\geq a})=
\mathcal K_{[a,1]}(\rsf F_{a}^{\geq a})=\mathcal K(\rsf F_{a})^{\geq
a}.
$$
It turns out that this is enough to give us out theorem.
\begin{thm}
    Let $\rsf F$ be any prime filter and $a\notin\rsf F$. Then
    $$
	\mathcal K(\rsf F)=\mathcal K(\rsf F_{a}).
    $$
\end{thm}
\begin{proof}
    Let $z\in\mathcal K(\rsf F)$. Then $z\join a\in\mathcal K(\rsf 
    F)\cap[a,1]=\mathcal K(\rsf F_{a})\cap[a,1]$.
    As $\mathcal K(\rsf F_{a})$ is prime and $a\notin\rsf F_{a}$ (and 
    hence not in $\mathcal K(\rsf F_{a})$) we must have $z\in\mathcal 
    K(\rsf F_{a})$.
    
    The converse works as $(\rsf F_{a})_{a}=\rsf F$.
\end{proof}

\section{Ordering on Filters}
In this section we show that the order on $\Set{\rsf F_{a} | 
a\in\mathcal L}$ is naturally isomorphic to a lower interval of 
$\mathcal L/\mathcal K(\rsf F)$.

Now let us fix on a prime filter $\rsf F$ and let $\rsf 
K=\mathcal K(\rsf F)$. 
Let $\eta$ be the canonical epimorphism from $\mathcal L$ to 
$\mathcal L/\rsf K$.

\begin{prop}
    $$
	\left(\rsf F/\rsf K\right)_{\eta(a)}=\rsf F_{a}/\rsf K.
    $$
\end{prop}
\begin{proof}
    \begin{align*}
        [x]\in \left(\rsf F/\rsf K\right)_{\eta(a)} & \iff 
	[x]\to[a]=[x\to a]\notin\rsf F/\rsf K  \\
         & \iff x\to a\notin\rsf F &&\text{ as }\rsf K=\mathcal 
	 K(\rsf F)  \\
         & \iff x\in\rsf F_{a}\\
         & \iff [x]\in \rsf F_{a}/\rsf K &&\text{ as }\rsf K=\mathcal 
	 K(\rsf F_{a}).
    \end{align*}
\end{proof}

\begin{cor}
    $$
	\rsf F_{a}=\eta^{-1}[\left(\rsf F/\rsf K\right)_{\eta(a)}]
    $$
\end{cor}
\begin{proof}
    As 
    $\rsf F_{a}=\eta^{-1}[\rsf F_{a}/\rsf K]=\eta^{-1}[\left(\rsf 
    F/\rsf K\right)_{\eta(a)}]$.
\end{proof}

\begin{cor}
    If $\eta(a)\le\eta(b)$ then $\rsf F_{b}\subseteq\rsf F_{a}$, with 
    equality if $\eta(a)=\eta(b)$.
\end{cor}
\begin{proof}
    Equality is immediate from the last corollary. 
    The inequality follows from \propref{prop:subOrder}.
\end{proof}

We aim to show  that the converse of this theorem is also true and so
get a complete description
of the ordering on the set $\Set{\rsf F_a| a\in\mathcal L}$.
First we have the easy case.
\begin{prop}\label{prop:easyCase}
	If $\rsf F_a=\rsf F_0$ then $\eta(a)=\eta(0)$.
\end{prop}
\begin{proof}
	$\eta(a)=\eta(0)$ iff $\lnot a\in\mathcal K(\rsf F)$. Suppose that 
	$\rsf F_0=\rsf F_a$ but $\lnot a\notin\mathcal K(\rsf F)$. Then
	there is some $b\notin\rsf F$ with $\lnot a\to b\in\rsf F$, ie $\lnot
	b\to a\in\rsf F$ and so $\lnot b\notin\rsf F_a=\rsf F_0$. Hence $b$
	must be in $\rsf F$ -- contradiction.
\end{proof}

\begin{lem}
	Let $\rsf F_a\subseteq\rsf F_b$. Then 
	\begin{enumerate}[(i)]
		\item $\rsf F_b=\rsf F_{a\meet b}$ and 
		\item $\rsf F_a=\rsf F_{a\join b}$.
	\end{enumerate}
\end{lem}
\begin{proof}
	\begin{enumerate}[(i)]
		\item As $a\meet b\le b$ we have $\rsf F_b\subseteq\rsf F_{a\meet
		b}$.
		
		If $y\to(a\meet b)= (y\to a)\meet (y\to b)\notin\rsf F$ then at
		least one of 
		$y\to a$ or $y\to b$ is not in $\rsf F$ and so $y\in\rsf
		F_a\cap\rsf F_b= \rsf F_b$.
		
		\item As $a\le a\join b$ we have $\rsf F_{a\join b}\subseteq\rsf
		F_a$.
		
		If $y\in\rsf F_a$ then $y$ is also in $\rsf F_b$ and so $y\to a$
and $y\to b$ are not in $\rsf F$. As $\rsf F$ is prime, we have
$y\to(a\join b)= (y\to a)\join(a\to b)\notin\rsf F$.
	\end{enumerate}
\end{proof}

\begin{thm}\label{thm:etaK}
	Let $\rsf F$ be a prime filter with $a,b\notin\rsf F$.
	\begin{enumerate}[(i)]
		\item If $a\le b$ and $\rsf F_a=\rsf F_b$ then $\eta(a)=\eta(b)$.
		
		\item If $\rsf F_a=\rsf F_b$ then $\eta(a)=\eta(b)$.
		
		\item If $\rsf F_a\subseteq\rsf F_b$ then $\eta(a)\geq\eta(b)$.
	\end{enumerate}
\end{thm}
\begin{proof}
	\begin{enumerate}[(i)]
		\item Given $a\le b$ and $\rsf F_a=\rsf F_b$ we also have 
		$\rsf F_a^{\geq a}=\rsf F_b^{\geq a}$. Now we can apply
		\propref{prop:easyCase}
		on $[a,1]$ to get $\eta_{\mathcal K_{[a,1]}}(\rsf F)(a)=
		\eta_{\mathcal K_{[a,1]}}(\rsf F)(b)$
		and so $b\to a\in \mathcal K_{[a,1]}(\rsf F^{\geq a})= \mathcal
		K(\rsf F)^{\geq a}\subseteq\mathcal K(\rsf F)$ and so
		$\eta(a)=\eta(b)$.
		
		\item As $\rsf F_a=\rsf F_{a\meet b}=\rsf F_b$ we can take part (i)
		with $a\meet b$ and $a$ and then with $a\meet b$ and $b$ to get
		$\eta(a)=\eta(a\meet b)=\eta(b)$.
		
		\item As $\rsf F_a=\rsf F_{a\join b}$ we have $\eta(a)= \eta(a\join
b)= \eta(a)\join\eta(b)$ and 
		so $\eta(a)\geq \eta(b)$.
	\end{enumerate}
\end{proof}

We note that if $P$ is a prime implication filter then 
$P_{a}=\eta_{P}^{-1}[][a]_{P}, \one]]$ for all $a\notin P$ -- as
$x\in P_{a}$ iff $x\to a\notin P$ iff $[x]\to[a]<\one$ iff 
$[a]<[x]$. Thus the set of all $\Set{P_{a} | a\in\mathcal L}$ is a 
copy of $\mathcal L/P$. 

\section{Kernels and Joins}
There are a number of other naturally occurring filters whose kernel
we can compute.
\begin{defn}
	Let $\rsf F$ be any filter, $P$ any implication filter. Then
	$$
		J(\rsf F,P)=\eta_P^{-1}[\rsf F/P].
	$$
\end{defn}

It is clear that both $\rsf F$ and $P$ are contained in $J(\rsf
F,P)$. 
\begin{prop}
	If $P$ and $Q$ are two implication filters then 
	$$
		P\join Q= J(P,Q).
	$$
\end{prop}
\begin{proof}
	As noted above we have $P\join Q\subseteq J(P,Q)$.
	
	Let $x\in J(P,Q)$. Then there is some $p\in P$ with $x\sim p\mod Q$.
As this implies
	$x=x\join x\sim p\join x\mod Q$ we will assume that $x\le p$ and so
$p\to x\in Q$. 
	Now we have 
	$x= x\meet p= p\otimes(p\to x)\in P\join Q$. 
\end{proof}
Thus $J$ is a generalization of join of filters -- one of several
possible generalizations.

We will consider the case that $P$ is a prime implication filter.
$J(\rsf F,P)$ is the preimage of a prime filter in $\mathcal
L/P$ and so must be prime. Therefore $\mathcal K(J(\rsf F,P))$ is
also a prime implication filter that must contain $P$. 

We now show that $\mathcal K(\rsf F)$ is contained in $\mathcal K(J(\rsf F,P))$.
\begin{prop}
	$\mathcal K(\rsf F)\subseteq\mathcal K(J(\rsf F,P))$.
\end{prop}
\begin{proof}
	Let $x\in J$ and $z\in\mathcal K(\rsf F)$. Then there is some
$f\in\rsf F$ with 
	$f\sim x\mod P$ and so $z\otimes f\sim z\otimes x$. As $z\otimes
f\in\rsf F$ we must have 
	$z\otimes x\in J$. 
\end{proof}

Hence
$$
	\mathcal K(\rsf F)\join P\subseteq\mathcal K(J(\rsf F,P)).
$$
We aim to show that these two sets are actually equal,  provided 
$J(\rsf F,P)\not=\mathcal L$. 

\begin{lem}
	Let $\rsf F$ be any order filter and $Q$ any implication filter. Then
	$$
		\eta_Q^{-1}[\mathcal K(\rsf F/Q)]=\mathcal K(\eta_Q^{-1}[\rsf F/Q]).
	$$
\end{lem}
\begin{proof}
	Let $[x]\in\mathcal K(\rsf F/Q)$ and $y\in\eta_Q^{-1}[\rsf F/Q]$. Then 
	$[y]\in\rsf F/Q$ so that $[x\otimes y]= [x]\otimes[y]\in\rsf F/Q$. 
	Hence $x\otimes y\in \eta_Q^{-1}[\rsf F/Q]$. As this holds for all such $y$ we have 
	$x\in\mathcal K(\eta_Q^{-1}[\rsf F/Q])$.
	
	Conversely, if $x\in \mathcal K(\eta_Q^{-1}[\rsf F/Q])$ and $[y]\in\rsf F/Q$ then
	$y\in\eta_Q^{-1}[\rsf F/Q]$ so that $x\otimes y\in\eta_Q^{-1}[\rsf F/Q]$
	and hence $[x]\otimes[y]= [x\otimes y]\in\rsf F/Q$. Hence
	$[x]\in\mathcal K(\rsf F/Q)$.
\end{proof}

There are two cases to our problem:
\begin{enumerate}[{Case }1:]
	\item $P\subseteq\mathcal K(\rsf F)$; and 
	\item $P\not\subset\mathcal K(\rsf F)$.
\end{enumerate}

\subsection{Case 1}
If $P\subseteq\mathcal K(\rsf F)$ then $J(\rsf F, 
P)=\eta_{P}^{-1}[\rsf F/P]= \rsf F$ by \thmref{thm:inclK}.

\subsection{Case 2}
Let $P'=P\join \mathcal K(\rsf F)$. This is strictly bigger than $\mathcal K(\rsf F)$
and so $\eta_{P'}^{-1}[\rsf F/P']\not=\rsf F$. This means there is a coset $[\ell]$ with
$\rsf F\cap[\ell]\not=\emptyset\not=[\ell]\setminus\rsf F$ and so 
$\rsf F/P'$ is actually principal, with lower endpoint 
$[\ell]$. 

As $J\not=\mathcal L$ we know that $[\ell]>[0]$ and 
the kernel of a nontrivial principal filter in a linear order is $\Set\one$. Hence
$$
	\mathcal K(\eta_{P'}^{-1}[\rsf F/P'])= \eta_{P'}^{-1}[\mathcal K(\rsf F/P')]= \eta_{P'}[\Set\one]= P'.
$$

Now we recall the natural isomorphism between 
$\left(\mathcal L/P\right)/\left(P'/P\right)$ and $\mathcal L/P'$ 
that restricts to an isomorphism on $\rsf F$. We will use this to 
establish our result.

\begin{lem}
    $P'/P\subseteq\mathcal K(\rsf F/P)$.
\end{lem}
\begin{proof}
    Let $[q]\in P'/P$ and $[f]\in\rsf F/P$ for some $f\in \rsf F$. 
    Then we have $q\sim k\meet p\sim k\mod P$ for some $k\in\mathcal 
    K(\rsf F)$ and so 
    $[q]\otimes[f]= [k]\otimes[f]= [k\otimes f]\in\rsf F/P$.
\end{proof}

It follows from this lemma that
$$
\eta_{P'/P}^{-1}[\left(\rsf F/P\right)/\left(P'/P\right)]= \rsf F/P
$$
but we know that $\left(\rsf F/P\right)/\left(P'/P\right)$ is 
essentially $\rsf F/P'$ and so principal with kernel $\Set\one$. 
Hence we have 
$$
\mathcal K_{\mathcal L/P}(\rsf F/P)=
\mathcal K_{\mathcal L/P}(\eta_{P'/P}^{-1}[\left(\rsf 
F/P\right)/\left(P'/P\right)])=
\eta_{P'/P}^{-1}[\mathcal K_{\mathcal L/P'}(\rsf F/P')]= P'/P
$$
and finally we have 
\begin{align*}
    \mathcal K(J(\rsf F, P)) & =\mathcal K(\eta_{P}^{-1}[\rsf F/P])  \\
     & = \eta_{P}^{-1}[\mathcal K(\rsf F/P)]  \\
     & = \eta_{P}^{-1}[P'/P]\\
     & = P'.
\end{align*}

The above results prove the theorem.
\begin{thm}
    Let $\rsf F$ be any prime lattice filter, and $P$ any prime 
    implication filter. Then
    $$
    \mathcal K(J(\rsf F, P))=\mathcal K(\rsf F)\join P.
    $$
\end{thm}

\subsection{Arbitrary Joins}

It is well known that every filter is the intersection of the prime 
filters that contain it. We can refine this somewhat and get an 
intersection to $\rsf F$ of a family of prime filters whose kernels intersect to 
$\mathcal K(\rsf F)$.

\begin{prop}
    Let $\mu S$ be the set of minimal prime filters. Then
    $$
	\rsf F=\bigcap_{m\in\mu S}J(\rsf F,  m).
    $$
\end{prop}
\begin{proof}
    We know that each $J(\rsf F,  m)$ is a prime filter containing 
    $\rsf F$ and so $\rsf F\subseteq$RHS.
    
    Conversely,  if $P$ is any prime filter containing $\rsf F$ then 
    there is a minimal prime filter $m\subseteq \mathcal 
    K(P)\subseteq P$. Hence
    $J(\rsf F, m)= \eta_{m}^{-1}[\rsf F/m]\subseteq 
    \eta_{m}^{-1}[P/m]= P$. 
    
    Therefore the given intersection is contained in $\bigcap\Set{ P 
    | \rsf F\subseteq P\text{ and }P\text{ prime}}$.
\end{proof}

\begin{cor}
    Let $\rsf F$ be any filter. Then 
    $$
	\mathcal K(\rsf F)=\bigcap_{m\in\mu S}\mathcal K(J(\rsf F,  m)).
    $$
\end{cor}
\begin{proof}
    This is a special case of the proposition as 
    $\mathcal K(J(\rsf F,  m))= \mathcal K(\rsf F)\join m= J(\mathcal 
    K(\rsf F), m)$.
\end{proof}


\begin{thebibliography}{9}
\bibitem{Martinez:one}
{Nestor G. Martinez}, 
\emph{A Simplified Duality Theorem for Implicative Lattices and 
$\ell$-Groups}, 
{Studia Logica} v56 (1996),  
pp {185--204}

\bibitem{Martinez:two}
{Nestor G. Martinez}, 
\emph{A topological duality for some lattice ordered algebraic structures 
including $\ell$-groups}, 
{Algebra Universalis} v31 (1994), 
pp {516--541}


\bibitem{MartPreist:one}
{Nestor G. Martinez and H.A. Priestley}, 
\emph{Uniqueness of MV-algebra implication and de Morgan negation},
journal={Mathware {\&} Soft Computing} v2 (1995), pp {229--245}

\bibitem{MainBk}{book}
{Mundici, D.}, {Cignoli, R.} and {D'Ottaviano, I.M.},
\emph{Algebraic Foundations of Many-valued Reasoning}, 
{Kluwer}, 
{2000}
    
\end{thebibliography}
\end{document}